\documentclass{amsart}[11pt]

\usepackage[left=1.25in, right=1.25in]{geometry}
\usepackage[all]{xy}
\usepackage{graphicx}
\usepackage{indentfirst}
\usepackage{bm}
\usepackage{amsthm}
\usepackage{mathrsfs}
\usepackage{latexsym}
\usepackage{amsmath}
\usepackage{amssymb}
\usepackage{color}

\begin{document}
\newtheorem{example}{Example}[section]
\newtheorem{theorem}{Theorem}[section]
\newtheorem{proposition}[theorem]{Proposition}
\newtheorem{corollary}[theorem]{Corollary}
\newtheorem{definition}{Definition}[section]
\newtheorem{lemma}[theorem]{Lemma}
\newtheorem{remark}{Remark}[section]
\newtheorem{question}{Question}[section]
\newtheorem{conjecture}{Conjecture}[section]
\newtheorem*{ac}{Acknowledgements}



\title{\textbf{M\"{o}bius Disjointness for C$^*$ Algebras}\\
} 

\author{\textsc{Jinsong Wu and Wei Yuan} \\
{}} %

\date{}

\maketitle



\begin{abstract}
In this paper we introduce the M\"{o}bius disjointness for C$^*$-algebras with their automorphisms and studied the M\"{o}bius disjointness for finite dimensional C$^*$ algebras, finite von Neumann algebras, reduced free group algebra and canonical anticommutation relation algebra etc.
\end{abstract}

\hspace*{3,6mm}\textit{Keywords:} M\"{o}bius function, Singular spectrum, Bogoliubov automorphism, Irrational rotation.
\vspace{30pt} 

{\small} {\small \textbf{2010 MR Subject Classification 46L55} }

\section{Introduction}
Let $\mu$ be the M\"{o}bius function, that is $\mu(n)$ is $0$ if $n$ is not square-free, and is $(-1)^t$ if $n$ is a product of $t$ distinct primes. We say $\mu$ is linearly disjoint from a flow $(X,T)$, where $X$ is compact topological space and $T:X\to X$ is a continuous map, if
\begin{equation}\label{disjoint}
\frac{1}{N}\sum_{n\leq N}\mu(n)f(T^nx)\to 0, \quad \mbox{as }N\to\infty
\end{equation}
for any $f$ in $C(X)$ and $x$ in $X$. The M\"{o}bius disjointness conjecture posted by Peter Sarnak (which is also known as Sarnak's conjecture) states that $\mu$ is linearly disjoint from every flow $(X,T)$ when the topological entropy of $(X,T)$ is zero. Recently, B. Green and T.Tao \cite{GreenTao12} showed that the M\"{o}bius function is strongly asymptotically orthogonal to any polynomial nilsequences; Liu and Sarnak \cite{LiuSar13} showed that the M\"{o}bius function is linearly disjoint from an analytic skew product on the 2-torus with additional conditions on the Fourier coefficients.

In the view of operator algebra, the algebra $C(X)$ of all continuous function on a compact Hausdorff space is a C$^*$ algebra, any point $x$ gives a pure state $\rho_x$ on $C(X)$, and the continuous map $T:X\to X$ will induce an endomorphism $\alpha_T$ on $C(X)$. Then the condition (\ref{disjoint}) can be rephrased as
\begin{equation}\label{disjoint1}
\frac{1}{N}\sum_{n\leq N}\mu(n)\rho_x(\alpha_T^n(f))\to 0, \quad \mbox{as }N\to\infty
\end{equation}
for any $f$ in $C(X)$ and $x$ in $X$. Hence it is reasonable to consider the
M\"{o}bius disjointness for a C$^*$ algebra with its endomorphism.

Let $\mathfrak{A}$ be a C$^*$-algebra and $\alpha$ its endomorphism, The pair $(\mathfrak{A},\alpha)$ is said to be a noncommutative flow. When $\mathfrak{A}$ is a von Neumann algebra, we need its endomorphism $\alpha$ to be weak-operator continuous. We then can define M\"{o}bius disjointness for $(\mathfrak{A},\alpha)$. The M\"{o}bius function $\mu$ is linearly disjoint from $(\mathfrak{A},\alpha)$ if
\begin{equation}\label{disjoint2}
\frac{1}{N}\sum_{n\leq N}\mu(n)\rho(\alpha^n(A))\to 0, \quad \mbox{as }N\to\infty
\end{equation}
for any state $\rho$ on $\mathfrak{A}$ and $A$ in $\mathfrak{A}$.

The M\"{o}bius disjointness conjecture also concerns the topological entropy for $(X,T)$. There are several way to define an entropy for a C$^*$-algebra with an automorphism (or endomorphism). In \cite{CS75}, A. Connes and E. St{\o}ruer introduced the entropy $h_\tau(\alpha)$ of automorphism $\alpha$ of II$_1$ von Neumann algebra $\mathcal{M}$ with a tracial state $\tau$. Later, A. Connes, H. Narnhofer and W. Thirring \cite{CNT} defined a dynamic entropy $h_\varphi(\alpha)$ for any C$^*$ algebra $\mathfrak{A}$ (or von Neumann algebra ) with its automorphism $\alpha$ with respect to a $\alpha$-invariant state $\varphi$ of $\mathfrak{A}$. This entropy for C$^*$ algebras is called CNT entropy. In 1995 D. Voiculescu \cite{V95} introduced topological entropy $ht(\alpha)$ for a nuclear C$^*$ algebra $\mathfrak{A}$ with its automorphism $\alpha$. N. Brown \cite{Brown} showed that the entropy given by Voiculescu is good for exact C$^*$ algebra and Dykema \cite{Dyk} shows that $h_\varphi(\alpha)\leq ht(\alpha)$ for an exact C$^*$ algebra $\mathfrak{A}$ with its automorphism $\alpha$, where $\varphi$ is an $\alpha$-invariant state on $\mathfrak{A}$. The entropy is called Voiculescu-Brown entropy. We will show that  the M\"{o}bius function is linearly disjoint from any factor of type II$_1$ with its automorphism when the state $\rho$ in the equation (\ref{disjoint2}) is normal, but its CNT entropy might be greater than zero. Hence we will pick Voiculescu-Brown entropy for exact C$^*$-algebras. As N. Brown pointed out in \cite{Brown}, the Voiculescu-Brown entropy is also well-defined for endomorphisms of exact C$^*$ algebras.

With M\"{o}bius disjointness and Voiculescu-Brown entropy for a noncommutative flow $(\mathfrak{A},\alpha)$, we formulate M\"{o}bius disjointness conjecture for exact C$^*$ algebras.

\begin{conjecture}\label{Con1}
The M\"{o}bius function $\mu$ is linearly disjoint from a noncommutative flow $(\mathfrak{A},\alpha)$ when the Voiculescu-Brown entropy of $(\mathfrak{A},\alpha)$ is zero, where $\mathfrak{A}$ is a unital exact C$^*$-algebra and $\alpha$ is an endomorphism of $\mathfrak{A}$.
\end{conjecture}
It is clear that the M\"{o}bius disjointness conjecture for exact C$^*$ algebra implies the M\"{o}bius disjointness conjecture. Due to Hanfeng Li, we see that The M\"{o}bius disjointness conjecture implies the M\"{o}bius disjointness conjecture for exact C$^*$ algebras when the endomorphism is injective.

In this paper, we will show that the M\"{o}bius disjointness conjecture for exact C$^*$ algebra is true when
\begin{itemize}
\item[(1)] $(\mathfrak{A},\alpha)$, $\mathfrak{A}$ is finite-dimensional C$^*$ algebra and $\alpha$ is any automorphism of $\mathfrak{A}$; $\mathfrak{A}$ is the unitalization of algebra of all compact operators on a Hilbert space with its automorphism whose entropy is zero;
\item[(2)] $(CAR(\mathcal{H}),\alpha_U)$, $CAR(\mathcal{H})$ is the canonical anti-commutation relation (briefly, CAR) algebra with respect to a Hilbert space $\mathcal{H}$ and $\alpha$ is the Bogoliubov automorphism induced by a unitary operator whose spectral measure is pure point measure on $\mathcal{H}$;
\item[(3)] $(C_r^*(F_\mathbb{Z}),\alpha)$, $C_r^*(F_\mathbb{Z})$ is the reduced C$^*$ algebra of free group on $\mathbb{Z}$ generators and $\alpha$ is the shift on generators.
\end{itemize}

The rest of paper will contain four sections. In Section 2 we will introduce some basic properties for M\"{o}bius disjointness for C$^*$ algebras and show that the M\"{o}bius disjointness conjecture for exact C$^*$ algebra is equivalent to the classical one when the endomorphism is injective.
In Section 3 the M\"{o}bius disjointness for finite-dimensional C$^*$ algebras and unitalization of algebra of compact operators is studied.
In Section 4 we studied the weakly linear disjointness for finite von Neumann algebras.
In Section 5
we will studied the M\"{o}bius disjointness for CAR algebras with Bogoliubov automorphisms
In Section 6 we show that M\"{o}bius function is linearly disjoint from  reduced free group C$^*$ algebra with its shift on generators.

\begin{ac}
The authors would like to thank Liming Ge for suggestion of considering Sarnak's conjecture in operator algebras. The first author also would like to thank E.H.El Abdalaoui for useful suggestions. We thanks Hanfeng Li for pointing out that the M\"{o}bius conjecture for exact C$^*$ algebra is equivalent to the classical one.
\end{ac}

\section{Prelimilaries}
We begin by recalling the definition of Voiculescu-Brown entropy.

Let $\mathfrak{A}$ be an exact (equivalently , nuclearly embeddable) C$^*$ algebra and $\alpha$ an automorphism of $\mathfrak{A}$. Let $\pi:\mathfrak{A}$ be a faithful *-representation. For a finite set $\Omega\subset\mathfrak{A}$ and $\delta>0$ we denote by $CPA(\pi,\Omega,\delta)$ the collection of triples $\varphi,\psi,\mathfrak{B}$ where $\mathfrak{B}$ is a finite dimensional C$^*$ algebra and $\varphi:\mathfrak{A}\to\mathfrak{B}$ and $\psi:\mathfrak{B}\to\mathcal{B}(\mathcal{H})$ are contractive completely positive maps such that $\|(\psi\circ\varphi)(a)-\pi(a)\|<\delta$ for all $a\in\Omega$. This collection is nonempty by nuclear embeddability. We define $rcp(\Omega,\delta)$ to be the infimum of $rank\mathfrak{B}$ over all $(\varphi,\psi,\mathfrak{B})\in CPA(\pi,\Omega,\delta)$ with rank referring to the dimension of a maximal abelian C$^*$ subalgebra. We then set
\begin{eqnarray*}
ht(\alpha,\Omega,\delta)&=&\limsup_{n\to\infty}\frac{1}{n}\log rcp(\Omega\cup \alpha\Omega\cup\cdots\cup\alpha^{n-1},\delta)\\
ht(\alpha,\Omega)&=&\sup_{\delta>0}ht(\alpha,\Omega,\delta),\\
ht(\alpha)&=&\sup_\Omega ht(\alpha,\Omega)
\end{eqnarray*}
where the last supremum is taken over all finite sets $\Omega\subset\mathfrak{A}$. The quantity $ht(\alpha)$ is a C$^*$-dynamical invariant which is the Voiculescu-Brown entropy of $\alpha$.

For the definition of Connes-Narnhofer-Thirring entropy (briefly CNT entropy) for C$^*$ algebras, we skip it here and refer to \cite{CNT}. Let $\mathfrak{A}$ be a unital nuclear C$^*$ algebra and $\alpha$ its automorphism. Brown \cite{Brown} showed that $ht(\alpha)=h_\varphi(\alpha)$, where $\varphi$ is $\alpha$-invariant state on $\mathfrak{A}$ and $h_\varphi(\alpha)$ is the CNT entropy of $\alpha$ with respect to $\varphi$. In general $ht(\alpha)\geq h_\varphi(\alpha)$.

Next, we will introduce the M\"{o}bius disjointness for a C$^*$ algebra $\mathfrak{A}$ with an endomorphism $\alpha$. We begin with the definition of the M\"{o}bius function $\mu$.

The M\"{o}bius function $\mu(n)$, $n = 1, 2, 3, \ldots$ is defined by

\begin{align*}
   \mu(n) = \left\{
  \begin{array}{l l}
    1 & \quad \text{if n=1 }\\
    0 & \quad \text{if n is not square free}\\
    (-1)^t & \quad \text{if n is a product of t distinct primes}
  \end{array} \right.
\end{align*}

Throughout the paper, the function $e(x)$ will denote $e^{2\pi i x}$ for real $x$.

\begin{lemma}\label{mu1}
Let $p$ be positive integer and $0\leq l<p$. For arbitrary $h>0$,
\begin{equation*}
\frac{1}{N}\sum_{n\leq N\atop n\equiv l(\mod p)}\mu(n)e(a_dn^d+a_{d-1}n^{d-1}+\cdots+a_1n+a_0)=O((\log N)^{-h})
\end{equation*}
where the implied constant may depend on $h$ and $p$, but is independent of any of coefficients $a_d,\cdots,a_0$. In particular,
\begin{equation*}
\frac{1}{N}\sum_{n\leq N}\mu(n)e(\theta n)=O((\log N)^{-h})
\end{equation*}
uniformly in $\theta$
\end{lemma}
This second equation in the lemma is proved by Davenport\cite{Davenport37}. The general case is proved by Hua\cite{Hua}.

\begin{proposition}\label{ds1}
Let $f:\mathbb{N}\to\mathbb{C}$ with $|f(n)|\leq 1$ for all $n\in\mathbb{N}$ and $\epsilon>0$ small enough. Assume that for all primes $p_1,p_2\leq e^{1/\epsilon}$, $p_1\neq p_2$, we have that for $M$ large enough
$$|\sum_{m\leq M}f(p_1m)\overline{f(p_2m)}|\leq \epsilon M.$$
Then for $N$ large enough
$$|\sum_{n\leq N}\mu(n)f(n)|\leq 2\sqrt{\epsilon\log \frac{1}{\epsilon}} N.$$
\end{proposition}
This is Theorem 2 in \cite{BSZ}.

Let $\mathfrak{A}$ be a C$^*$ algebra and $\alpha$ an automorphism of $\mathfrak{A}$. The pair $(\mathfrak{A},\alpha)$ is a noncommutative flow.
The M\"{o}bius function $\mu$ is linearly disjoint from $(\mathfrak{A},\alpha)$ if
$$\frac{1}{N}\sum_{n=1}^N\mu(n)\rho(\alpha^n(A))\to 0\quad \mbox{as }N\to\infty$$
for every $A$ in $\mathfrak{A}$ and any state $\rho$ on $\mathfrak{A}$.

Similarly, we can define M\"{o}bius disjointness for von Neumann algebras. Let $\mathcal{M}$ be a von Neumann algebra and $\alpha$ is an automorphism of $\mathcal{M}$, we say $\mu$ is weakly linearly disjoint from $(\mathcal{M},\alpha)$ if
$$\frac{1}{N}\sum_{n=1}^N\mu(n)\rho(\alpha^n(A))\to 0\quad \mbox{as }N\to\infty$$
for every $A$ in $\mathcal{M}$ and any normal state $\rho$ on $\mathcal{M}$.

We will first state that the M\"{o}bius disjointness conjecture for C$^*$ algebra is equivalent to the M\"{o}bius disjointness conjecture.

\begin{proposition}
The M\"{o}bius disjointness conjecture implies the M\"{o}bius disjointness conjecture for exact C$^*$ algebras when the endomorphism is injective.
\end{proposition}
\begin{proof}
Suppose that the M\"{o}bius disjointness conjecture is true. Let $(\mathfrak{A},\alpha)$ be a noncommutative flow, where $\mathfrak{A}$ is a unital C$^*$ algebra and $\alpha$ is an endomorphism of $\mathfrak{A}$. Suppose that the Voiculescu-Brown entropy $ht(\alpha)=0$. Denote by $f_\alpha$ the continuous map from the state space $\mathcal{S}(\mathfrak{A})$ of $\mathfrak{A}$ to itself induced by $\alpha$ given by $f_\alpha(\rho)=\rho\circ\alpha$.

First, we assume that $\alpha$ is an automorphism. By \cite{Kerr04} or \cite{KerrLi05}, we have that the topological entropy $h_{\text{top}}(f_\alpha)=0$. From the assumption that M\"{o}bius disjointness conjecture is true, we see that the display (\ref{disjoint2}) holds.

Next, we assume that $\alpha$ is injective endomorphism. Denote by $\mathfrak{B}$ the inductive limit of the system $\{\mathfrak{A}_n\}_{n\geq 1}$, where $\mathfrak{A}_n=\mathfrak{A}$ with *-isomorphism $\phi_{m,n}=\alpha^{n-m}:\mathfrak{A}_m\to\mathfrak{A}_n$ when $m\leq n$. Deonte by $\beta$ the induced automorphism of $\alpha$ on $\mathfrak{B}$. Then by \cite{Brown}, Proposition 2.14, we have that $ht(\beta)=0$. Hence the induced continuous map $f_\beta$ on $\mathcal{S}(\mathfrak{B})$ has topological entropy $h_{\text{top}}(f_\beta)=0$. We have the display (\ref{disjoint2}) holds for $\mathcal{S}(\mathfrak{B})$. Note that $\mathcal{S}(\mathfrak{B})$ is a projective limit of $\mathcal{S}(\mathfrak{A})$. We obtain that the display (\ref{disjoint2}) holds for $\mathcal{S}(\mathfrak{A})$.

\end{proof}

\begin{lemma}
Let $\mathcal{A}$ be a C$^*$ algebra and $\alpha\in Aut(\mathfrak{A})$. Suppose that $\alpha^q=Id$ for some $q$ in $\mathbb{N}$. Then $\mu$ is linearly disjoint from $(\mathcal{A},\alpha)$.
\end{lemma}
\begin{proof}
Lemma \ref{mu1} can be applied here directly.
\end{proof}

Thus we will focus on aperiodic automorphisms of C$^*$ algebras.

\begin{lemma}\label{dense}
Let $\mathfrak{A}$ be a C$^*$ algebra and $\alpha\in Aut(\mathfrak{A})$. Suppose that $\Omega$ is a subset of $\mathfrak{A}$ such that $\mathfrak{A}$ is the norm closure of linear span of $\Omega$. If
\begin{equation*}
\frac{1}{N}\sum_{n=1}^N\mu(n)\rho(\alpha^n(A))\to 0\quad \mbox{as }N\to\infty
\end{equation*}
for any state $\rho$ on $\mathfrak{A}$ and $A$ in $\Omega$, then $\mu$ is linearly disjoint from $(\mathfrak{A},\alpha)$.
\end{lemma}
\begin{proof}
For any $A$ in $\mathfrak{A}$, any $\epsilon>0$, there exist $\lambda_1,\ldots,\lambda_k\in\mathbb{C}$ and $A_1,\ldots,A_k\in\Omega$ such that $\|A-\sum_{j=1}^k\lambda_jA_j\|<\epsilon/2$. Since $A_1,\ldots, A_k$ satisfies the property list in the lemma, there exists a $N_0\in\mathbb{N}$ such that for any $N\geq N_0$, we have
$$\left|\frac{1}{N}\sum_{n=1}^N\mu(n)\rho(\alpha^n(A_j))\right|<\frac{1}{2\sum_{j=1}^k|\lambda_j|}\epsilon, \quad j=1,\ldots,k.$$
Then we have that
\begin{equation*}
\begin{split}
&\left|\frac{1}{N}\sum_{n=1}^N\mu(n)\rho(\alpha^n(A))\right|\\
\leq &\left|\frac{1}{N}\sum_{n=1}^N\mu(n)\rho(\alpha^n(A-\sum_{j=1}^k\lambda_jA_j))\right|+\left|\frac{1}{N}\sum_{n=1}^N\mu(n)\sum_{j=1}^k\lambda_j\rho(\alpha^n(A_j))\right|\\
\leq &\epsilon/2+\sum_{j=1}^k|\lambda_j|\left|\frac{1}{N}\sum_{n=1}^N\mu(n)\rho(\alpha^n(A_j))\right|\leq \epsilon/2+\epsilon/2=\epsilon.
\end{split}
\end{equation*}
This shows that $\mu$ is linearly disjoint from $(\mathfrak{A},\alpha)$.
\end{proof}

\begin{lemma}\label{sub}
Let $\mathfrak{A}$ be a C$^*$ algebra and $\alpha\in Aut(\mathfrak{A})$. Suppose that $\mathfrak{A}_0$ is a C$^*$-subalgebra of $\mathfrak{A}$ such that $\alpha(\mathfrak{A}_0)=\mathfrak{A}_0$. If $\mu$ is linearly disjoint from $(\mathfrak{A},\alpha)$, then $\mu$ is linearly disjoint from $(\mathfrak{A}_0,\alpha|_{\mathfrak{A}_0})$.
\end{lemma}
\begin{proof}
By Hahn-Banach Theorem, any state $\rho$ of $\mathfrak{A}_0$ can be extended to a state $\rho'$ of $\mathfrak{A}$ such that $\rho'|_{\mathfrak{A}_0}=\rho$, we see the lemma holds.
\end{proof}

\begin{lemma}\label{cross}
Let $\mathfrak{A}$ be a C$^*$ algebra and $\alpha\in Aut(\mathfrak{A})$. Then $\mu$ is linearly disjoint from $(\mathfrak{A},\alpha)$ if and only if $\mu$ is linearly disjoint from $(\mathfrak{A}\rtimes_\alpha\mathbb{Z},Ad(U))$, where $U\in \mathfrak{A}\rtimes_\alpha\mathbb{Z}$ is the  unitary element implementing $\alpha$.
\end{lemma}
\begin{proof}
By Lemma \ref{sub}, we only have to show that if $\mu$ is linearly disjoint from $(\mathfrak{A},\alpha)$, then $\mu$ is linearly disjoint from $(\mathfrak{A}\rtimes_\alpha\mathbb{Z},Ad(U))$. So we assume that $\mu$ is linearly disjoint from $(\mathfrak{A},\alpha)$. Let $\Omega=\{AU^k|A\in\mathfrak{A}, k\in\mathbb{Z}\}$. Then the linear span of $\Omega$ is norm dense in $\mathfrak{A}$. By Lemma \ref{dense}, we only have to show
\begin{equation*}
\frac{1}{N}\sum_{n=1}^N\mu(n)\rho(U^n(AU^k)U^{*n})\to 0\quad \mbox{as }N\to\infty
\end{equation*}
Since $U^n(AU^k)U^{*n}=\alpha^n(A)U^k$ and $\rho(\cdot U^k)$ gives a bounded linear functional on $\mathfrak{A}$, we have the lemma proved.
\end{proof}

\begin{corollary}\label{inner}
If the conjecture \ref{Con1} is true for any $(\mathfrak{A},Ad(U))$, where $\mathfrak{A}$ is an exact C$^*$ algebra and $Ad(U)$ is an inner automorphism of $\mathfrak{A}$. Then the conjecture \ref{Con1} is true when automorphism is considered.
\end{corollary}
\begin{proof}
Let $\mathfrak{B}$ be an exact C$^*$ algebra and $\beta\in Aut(\mathfrak{B})$. Suppose that $ht_{\mathfrak{B}}(\beta)=0$. By \cite{Brown}, we have that $ht_{\mathfrak{B}\rtimes_\beta\mathbb{Z}}(Ad(U))=0$, where $U\in\mathfrak{B}\rtimes_\beta\mathbb{Z}$ is the unitary element which implements $\beta$. Then $\mu$ is linearly disjoint from $(\mathfrak{B}\rtimes_\beta\mathbb{Z},Ad(U))$. By Lemma \ref{cross}, we obtain $\mu$ is linearly disjoint from $(\mathfrak{B},\beta)$.
\end{proof}

\begin{lemma}\label{conjugate}
Let $\mathfrak{A}$ be a C$^*$ algebra and $\alpha,\beta\in Aut(\mathfrak{A})$. Then $\mu$ is linearly disjoint from $(\mathfrak{A},\alpha)$ if and only if $\mu$ is linearly disjoint from $(\mathfrak{A},\beta\alpha\beta^{-1})$.
\end{lemma}
\begin{proof}
Suppose that $\mu$ is linearly disjoint from $(\mathfrak{A},\alpha)$. Then
\begin{equation*}
\frac{1}{N}\sum_{n=1}^N\mu(n)\rho(\alpha^n(A))\to 0\quad \mbox{as }N\to\infty
\end{equation*}
for any $A$ in $\mathfrak{A}$ and state $\rho$ of $\mathfrak{A}$. Replace $\rho$ by $\rho\circ\beta$ and $A$ by $\beta^{-1}(A)$, we have that
\begin{equation*}
\frac{1}{N}\sum_{n=1}^N\mu(n)\rho(\beta\alpha^n(\beta^{-1}(A)))\to 0\quad \mbox{as }N\to\infty
\end{equation*}
for any $A$ in $\mathfrak{A}$ and state $\rho$ of $\mathfrak{A}$. Since $\beta\alpha^n\beta^{-1}=(\beta\alpha\beta^{-1})^n$, the lemma holds.
\end{proof}

\begin{lemma}
Let $\mathfrak{A}$ be a C$^*$ algebra and $\alpha\in Aut(\mathfrak{A})$. Suppose $\mathcal{I}$ is an ideal of $\mathfrak{A}$ such that $\alpha(\mathcal{I})=\mathcal{I}$. Let $\widetilde{\alpha}$ denote the induced automorphism on $\mathfrak{A}/\mathcal{I}$. If $\mu$ is linearly disjoint from $(\mathfrak{A},\alpha)$, then $\mu$ is linearly disjoint from $(\mathfrak{A}/\mathcal{I},\widetilde{\alpha})$.
\end{lemma}
\begin{proof}
Let $\Phi:\mathfrak{A}\to\mathfrak{A}/\mathcal{I}$ be the canonical quotient homomorphism. Then $\widetilde{\alpha}\circ\Phi=\Phi\circ\alpha$. For any $A$ in $\mathfrak{A}/\mathcal{I}$, there exists $A_0$ in $\mathfrak{A}$ such that $\Phi(A_0)=A$. Let $\rho$ be a state of $\mathfrak{A}/\mathcal{I}$. Then $\rho(\widetilde{\alpha}^n(A)) =\rho(\Phi(\alpha^n(A_0)))$. While $\rho\circ\Phi$ is a state on $\mathfrak{A}$, we see that the lemma holds.
\end{proof}

\begin{lemma}\label{state}
Let $\mathfrak{A}$ be a C$^*$ algebra and $\alpha\in Aut(\mathfrak{A})$. Let $\{\rho_m\}_m$ be a sequence of bounded linear functionals on $\mathfrak{A}$ such that $\|\rho_m-\rho\|\to 0$ as $m\to\infty$ for some bounded linear functional $\rho$ on $\mathfrak{A}$. If for $A$ in the unit ball of $\mathfrak{A}$ and any $m$
$$\frac{1}{N}\sum_{n\leq N}\mu(n)\rho_m(\alpha^n(A))\to 0\quad \mbox{as }N\to\infty$$
then
$$\frac{1}{N}\sum_{n\leq N}\mu(n)\rho(\alpha^n(A))\to 0\quad \mbox{as }N\to\infty.$$
\end{lemma}
\begin{proof}
We will show that for any $\epsilon>0$, there exists $N_0$ in $\mathbb{N}$ such that for $N>N_0$ we have $|\frac{1}{N}\sum_{n\leq N}\mu(n)\rho(\alpha^n(A))|<\epsilon$. Since $\|\rho_m-\rho\|\to 0$ as $m\to\infty$, there exist $m_0$ in $\mathbb{N}$ such that $\|\rho_m-\rho\|<\epsilon/2$ for $m\geq m_0$.
From the fact that $\frac{1}{N}\sum_{n\leq N}\mu(n)\rho_m(\alpha^n(A))\to 0\quad \mbox{as }N\to\infty$, we have there exists $N_0$ in $\mathbb{N}$ such that $|\frac{1}{N}\sum_{n\leq N}\mu(n)\rho_m(\alpha^n(A))|\leq \epsilon/2$ for $N\geq N_0$. Then for $N\geq N_0$
\begin{eqnarray*}
&&|\frac{1}{N}\sum_{n\leq N}\mu(n)\rho(\alpha^n(A))|\\
&=&|\frac{1}{N}\sum_{n\leq N}\mu(n)(\rho-\rho_m)(\alpha^n(A))|+|\frac{1}{N}\sum_{n\leq N}\mu(n)\rho_m(\alpha^n(A))|\\
&\leq &\epsilon/2+\epsilon/2=\epsilon.
\end{eqnarray*}
This proved the lemma.
\end{proof}

\begin{remark}
Let $\mathcal{A}$ be a C$^*$ algebra, $\alpha$ its automorphism, and $\varphi$ is an $\alpha$-invariant state on $\mathcal{A}$. Suppose that $\mu$ is linearly disjoint from $(\mathcal{A},\alpha)$. Let $\mathcal{M}=\pi_\varphi(\mathcal{A})''$ where $\pi_\varphi$ is the GNS representation arising from $\varphi$ and $\hat{\alpha}$ be the automorphism of $\mathcal{M}$ extended by $\alpha$. Then $\mu$ is weakly disjoint from $(\mathcal{M},\hat{\alpha})$.
\end{remark}

\section{Finite-dimensional C$^*$ Algebras}
In this section, we will show that the M\"{o}bius disjointness conjecture for exact C$^*$ algebra holds for finite-dimensional C$^*$ algebras with its automorphisms.

First we would like to give the matrix version for Lemma \ref{mu1}.
\begin{proposition}\label{matrix}
Let $p$ be positive integer and $0\leq l< p$. Let $U_1,\ldots,U_d$ be unitary matrices in $M_k(\mathbb{C})$, $A_1,\ldots,A_d$ matrices in $M_k(\mathbb{C})$ with $\|A_j\|\leq 1,j=1\ldots,d$, and $\phi_1,\ldots,\phi_d$ polynomials with real coeffiences  where $d$ is an positive integer. Then for arbitrary $h>0$,
$$\frac{1}{N}\sum_{n\leq N\atop n\equiv l(\mod p)}\mu(n)tr_k(U_1^{\phi_1(n)}A_1U_2^{\phi_2(n)}A_2\cdots U_d^{\phi_d(n)}A_d)= k^{d/2}O((\log N)^{-h})$$
uniformly in unitary elements $U_1,\ldots,U_d$, $A_1,\ldots,A_d$ and independent in coefficiences of $\phi_j$, $j=1,\ldots,d$, where $tr_k$ is the normalized trace of $M_k(\mathbb{C})$, the implied constant only depends on the maximal degrees of $\phi_j$s and p. If $d=2$, we have
$$\frac{1}{N}\sum_{n\leq N\atop n\equiv l(\mod p)}\mu(n)tr_k(U_1^{\phi_1(n)}A_1U_2^{\phi_2(n)}A_2)= O((\log N)^{-h})$$
which is independent of $k$.
\end{proposition}
\begin{proof}
We assume that $U_j=diag(e(\theta_1^{(j)}),\ldots,e(\theta_k^{(j)}))$ for $j=1,\ldots, d$ and $A_j=(a_{ts}^{(j)})$ for $j=1,\ldots,d$. Then
\begin{equation*}
\begin{split}
&tr_k(U_1^{\phi_1(n)}A_1U_2^{\phi_2(n)}A_2\cdots U_d^{\phi_d(n)}A_d)\\
=&tr_k((e(\theta_t^{(1)}\phi_1(n))a_{ts}^{(1)})_{t,s}\cdots(e(\theta_t^{(d)}\phi_d(n))a_{ts}^{(d)})_{t,s})\\
=&\frac{1}{k}\sum_{t_1,\ldots,t_d=1}^ke(\theta_{t_1}^{(1)}\phi_1(n)+\cdots+\theta_{t_d}^{(d)}\phi_d(n))a_{t_1t_2}^{(1)}\cdots a_{t_dt_1}^{(d)}.
\end{split}
\end{equation*}
Hence
\begin{equation*}
\begin{split}
&\frac{1}{N}\sum_{n\leq N}\mu(n)tr_k(U_1^{\phi_1(n)}A_1U_2^{\phi_2(n)}A_2\cdots U_d^{\phi_d(n)}A_d)\\
&=\frac{1}{N}\sum_{n\leq N}\mu(n)\frac{1}{k}\sum_{t_1,\ldots,t_d=1}^ke(\theta_{t_1}^{(1)}\phi_1(n)+\cdots+\theta_{t_d}^{(d)}\phi_d(n))a_{t_1t_2}^{(1)}\cdots a_{t_dt_1}^{(d)}\\
&=\frac{1}{k}\sum_{t_1,\ldots,t_d=1}^k\left(\frac{1}{N}\sum_{n\leq N}\mu(n)e(\theta_{t_1}^{(1)}\phi_1(n)+\cdots+\theta_{t_d}^{(d)}\phi_d(n))\right)a_{t_1t_2}^{(1)}\cdots a_{t_dt_1}^{(d)}
\end{split}
\end{equation*}
By Lemma \ref{mu1},
$$\frac{1}{N}\sum_{n\leq N}\mu(n)e(\theta_{t_1}^{(1)}\phi_1(n)+\cdots+\theta_{t_d}^{(d)}\phi_d(n))=O((\log N)^{-h})$$
which the implied constant only depends on the maximal degree of $\phi_j$s and $p$. Note that $$\frac{1}{k} \sum_{t_1,\ldots,t_d=1}^k |a_{t_1t_2}^{(1)} \cdots a_{t_dt_1}^{(d)}|$$ is the trace of $A_1'A_2'\cdots A_d'$, where $A_j'=(|a_{ts}^{(j)}|)_{t,s}$ and $\|A_j'\|\leq \sqrt{k}$. Therefore
\begin{equation*}
\frac{1}{k}\sum_{t_1,\ldots,t_d=1}^k |a_{t_1t_2}^{(1)}\cdots a_{t_dt_1}^{(d)}|=tr_k(A_1'A_2'\cdots A_d')\leq \|A_1'\|\cdots\|A_d'\|\leq k^{d/2}
\end{equation*}
This completes the proof for general $d$. If $d=2$, we have
\begin{equation*}
\frac{1}{k}\sum_{t_1,t_2=1}^k |a_{t_1t_2}^{(1)}a_{t_2t_1}^{(2)}|\leq \frac{1}{k}(\sum_{t_1,t_2=1}^k|a_{t_1,t_2}^{(1)}|^2)^{1/2}(\sum_{t_1,t_2=1}^k|a_{t_1,t_2}^{(2)}|^2)^{1/2}=\|A_1\|_2\|A_2\|_2\leq 1.
\end{equation*}
Hence we see that when $d=2$, it is independent of $k$.
\end{proof}

\begin{theorem}\label{matrixm}
Suppose that $\mathfrak{A}$ be a finite dimensional C$^*$ algebra and $\alpha\in Aut(\mathfrak{A})$. Then $ht(\alpha)=0$ and $\mu$ is linearly disjoint from $(\mathfrak{A},\alpha)$.
\end{theorem}
\begin{proof}
Let $\psi$ be a faithful $\alpha$-invariant state on $\mathfrak{A}$. By the Gelfand-Naimark-Segal construction, we assume that $\mathfrak{A}$ acts on finite-dimensional Hilbert space $L^2(\mathfrak{A},\psi)$ and $\alpha(A)=UAU^*$ for some unitary operator $U$ on $L^2(\mathfrak{A},\psi)$. Let $\tau$ be the normalized trace. For any state $\rho$ on $\mathfrak{A}$, by Hahn-Banach theorem, there is a state $\rho'$ extending $\rho$. Then we have to show
$$\frac{1}{N}\sum_{n\leq N}\mu(n)\tau(U^nAU^{*n}B)\to 0\quad \mbox{as }N\to\infty$$
for any $A$, $B$ in $\mathcal{B}(L^2(\mathfrak{A},\psi))$, but it is directly from Lemma \ref{matrix}. Hence $\mu$ is linearly disjoint from $(\mathfrak{A},\alpha)$.

By the definition of Voiculescu-Brown entropy for exact C$^*$ algebras, we see that $ht(\alpha)=0$.
\end{proof}

\begin{proposition}\label{CL2}
Let $\mathcal{H}$ be a Hilbert space and $\alpha$ is an automorphism of $\mathcal{C}(\mathcal{H})$, the algebra of all compact operators on $\mathcal{H}$. Then $\mu$ is linearly disjoint from $(\mathcal{C}(\mathcal{H})+\mathbb{C}I,\alpha)$.
\end{proposition}
\begin{proof}
We will show that for any automorphism $\alpha$ of $\mathcal{C}(\mathcal{H})$ there exists a unitary operator $U$ on $\mathcal{H}$ such that $\alpha=Ad(U)$. Let $\{E_{jk}\}_{j,k}$ be a system of matrix units for $\mathcal{C}(\mathcal{H})$. Since $E_{11},\alpha(E_{11})$ are equivalent in $\mathcal{B}(\mathcal{H})$, there exist a partial isometry $V$ such that $VV^*=E_{11}$ and $V^*V=\alpha(E_{11})$. Then a unitary operator $U$ implementing $\alpha$ can be given as $U=\sum_{j=1}^\infty E_{j1}V\alpha(E_{1j})$.

It is known that the dual of $\mathcal{C}(\mathcal{H})$ is the space of all normal linear functionals of $\mathcal{B}(\mathcal{H})$, i.e. the predual $\mathcal{B}(\mathcal{H})_*$ of $\mathcal{B}(\mathcal{H})$. Hence we can focus on vector state on $\mathcal{B}(\mathcal{H})$.

Since every operator is finite sum of positive operator and every positive compact operator can be approximated by finite linearly combination of rank one projections. By Lemma \ref{dense}, it suffices to show that
$$\frac{1}{N}\sum_{n\leq N}\mu(n)\langle U^{*n}P_\xi U^n\eta,\eta\rangle\to 0,\quad \mbox{as } N\to\infty,$$
for any unit vector $\xi,\eta$ in $\mathcal{H}$.

Meanwhile
\begin{equation*}
\begin{split}
&\left|\frac{1}{N}\sum_{n\leq N}\mu(n)\langle U^{*n}P_\xi U^n\eta,\eta\rangle\right|
=\left|\frac{1}{N}\sum_{n\leq N}\mu(n)\langle \langle U^n\eta,\xi\rangle \xi,U^n\eta\rangle\right|\\
&=\left|\frac{1}{N}\sum_{n\leq N}\mu(n)\langle U^n\eta,\xi\rangle \langle U^{*n}\xi,\eta\rangle\right|\\
&=\left|\frac{1}{N}\sum_{n\leq N}\mu(n)\langle (U^n\otimes U^{*n})\eta\otimes\xi, \xi\otimes\eta\rangle\right|\\
&=\left|\left\langle \frac{1}{N}\sum_{n\leq N}\mu(n) (U\otimes U^{*})^n\eta\otimes\xi, \xi\otimes\eta\right\rangle\right|\\
&\leq \max_{z\in\mathbb{T}}|\frac{1}{N}\sum_{n\leq N}\mu(n) z^n|=O((\log N)^{-h})
\end{split}
\end{equation*}
for any fixed $h>0$ by \cite{Davenport37} or Lemma \ref{mu1}.
\end{proof}

\begin{question}
Suppose $\mathfrak{A}$ is an AF-algebra with an inner automorphim $\alpha$ and $ht(\alpha)=0$. Is the M\"{o}bius function $\mu$ linear disjoint from $(\mathfrak{A},\alpha)$?
\end{question}

\section{Finite von Neumann Algebras}
In the commutative case, let $F_\nu=(X_F, T, \nu_F)$ such that $T$ is measure-preserving, it is pointed out by P. Sarnak in \cite{Sar1} that such disjointness (orthogonality) is valid universally, i.e. for every $F_\nu$. We will present a similar result in noncommutative case.

\begin{proposition}\label{H1}
Let $\mathcal{M}$ be a finite von Neumann algebra with a faithful normal tracial state $\tau$ and $U$ a unitary element in $\mathcal{M}$. Then $\mu$ is weakly linearly disjoint from $(\mathcal{M},Ad(U))$.
\end{proposition}
\begin{proof}
We assume that $\mathcal{M}$ acts standardly on the Hilbert space $L^2(\mathcal{M},\tau)$. To show $\mu$ is weakly linearly disjoint from $(\mathcal{M},Ad(U))$, we have to estimate
$$\frac{1}{N}\sum_{n\leq N}\mu(n)\rho(U^{*n}TU^{n}))$$
for any $T$ in the unit ball of $\mathcal{M}$ and any normal state $\rho$.
Let $\eta$ be the trace vector associated to $\tau$ in $L^2(\mathcal{M},\tau)$. By Lemma 2.10 in \cite{Haa75}, the fact that $\mathcal{M}\eta$ is dense in $L^2(\mathcal{M},\tau)$ and Lemma \ref{state}, we may assume that $\rho(\cdot)=\langle\cdot A\eta,A\eta\rangle$, where $\eta$ is the canonical trace vector in $L^2(\mathcal{M},\tau)$ and $A\in\mathcal{M}$.

By spectral decomposition theorem, for any $\epsilon>0$ there is a unitary element $V=\sum_{k=1}^m e(\theta_k)P_k$ in $\mathcal{M}$ such that $\|U^n-V^n\|\leq \epsilon$, where $P_k,k=1,\ldots,m$ are orthogonal projections and $n=1,\ldots,N$. We have
\begin{equation*}
\begin{split}
&\left|\frac{1}{N}\sum_{n=1}^N\mu(n)\langle TU^nA\eta,U^nA\eta \rangle\right|\\
\leq& \left|\frac{1}{N}\sum_{n=1}^N\mu(n)\langle T(U^n-V^n)A\eta,U^nA\eta\rangle\right|+\left|\frac{1}{N}\sum_{n=1}^N\mu(n)\langle TV^nA\eta,(U^n-V^n)A\eta\rangle\right|\\
&+\left|\frac{1}{N}\sum_{n=1}^N\mu(n)\langle TV^nA\eta,V^nA\eta\rangle\right|\\
\leq &2\epsilon+\left|\sum_{l,k}\frac{1}{N}\sum_{n=1}^N\mu(n)e(n(\theta_l-\theta_k))\langle TP_lA\eta,P_kA\eta\rangle\right|\\
=&2\epsilon+\left|\sum_{l,k}\frac{1}{N}\sum_{n=1}^N\mu(n)e(n(\theta_l-\theta_k))\tau(P_kTP_lAA^*)\right|\\
\leq & 2\epsilon+O((\log N)^{-h})\left|\sum_{l,k}\tau(P_kTP_lAA^*P_k)\right|\\
\leq & 2\epsilon+O((\log N)^{-h})\sum_{l,k}\|P_kTP_l\|_2\|P_kAA^*P_l\|_2\\
\leq & 2\epsilon+O((\log N)^{-h})(\sum_{l,k}\|P_kTP_l\|_2^2)^{1/2}(\sum_{l,k}\|P_kAA^*P_l\|_2^2)^{1/2}\\
=&2\epsilon+O((\log N)^{-h})\|T\|_2\|AA^*\|_2\leq 2\epsilon+O((\log N)^{-h})=O((\log N)^{-h})
\end{split}
\end{equation*}
Hence $\mu$ is weakly linearly disjoint from $(\mathcal{M},Ad(U))$.
\end{proof}

\begin{proposition}\label{H2}
Let $\mathcal{M}$ be a finite von Neumann algebra with a faithful normal tracial state $\tau$. Suppose $\alpha$ is a trace-preserving automorphism of $\mathcal{M}$. Then $\mu$ is weakly linearly disjoint from $(\mathcal{M},\alpha)$.
\end{proposition}
\begin{proof}
We assume that $\alpha^n\neq id$ for any nonzero $n\in\mathbb{Z}$. By considering the crossed product $\mathcal{M}\rtimes_\alpha\mathbb{Z}$ of $\mathcal{M}$ by $\mathbb{Z}$, we have that $\mathcal{M}\rtimes_\alpha\mathbb{Z}$ is a finite von Neumann algebra. Let $U$ be the unitary element in $\mathcal{M}\rtimes_\alpha\mathbb{Z}$ implementing $\alpha$. Then $AdU$ is an inner automorphism of $\mathcal{M}\rtimes_\alpha\mathbb{Z}$ and $\mu$ is weakly linearly disjoint from $(\mathcal{M}\rtimes_\alpha\mathbb{Z},AdU)$. Hence $\mu$ is weakly linearly disjoint from $(\mathcal{M},\alpha)$.
\end{proof}

\begin{remark}
In \cite{CS75}, Connes and St{\o}rmer showed that the dynamic entropy for finite von Neumann algebra of the shift automorphism of the hyperfinite factor of type II$_1$ is greater than zero. But the proposition \ref{H2} shows that the M\"{o}bious function is weakly linearly disjoint from the shifts. As the classical case, this is not our main interest.
\end{remark}

\begin{question}
Is $\mu$ weakly linearly disjoint from $(\mathcal{M},Ad(U))$, where $\mathcal{M}$ is a factor of type III and $U$ is a unitary element in $\mathcal{M}$?
\end{question}

\section{CAR Algebras}
In this section, we will investigate the M\"{o}bius disjointness for Canonical Anticommutation Relation algebras with Bogoliubov automorphisms.

Let $\mathcal{H}$ be a Hilbert space. The full Fock space $T(\mathcal{H})$ of $\mathcal{H}$ is given by
$$T(\mathcal{H})=\bigoplus_{n=0}^\infty\otimes^n\mathcal{H},$$
where $\otimes^0\mathcal{H}$ is $\mathbb{C}1$.

Define a linear map $P$ of $T(\mathcal{H})$ given by
$$P(\xi_1\otimes\cdots\otimes\xi_n)=\frac{1}{n!}\sum_{\sigma\in S_n}(-1)^\sigma\xi_{\sigma(1)}\otimes\cdots\otimes\xi_{\sigma(n)},$$
where $S_n$ is the permutation group on $\{1,2,\ldots,n\}$. One can check that $P$ is a projection on $T(\mathcal{H})$. We denote by $\Lambda^n\mathcal{H}$ the Hilbert subspace $P(\otimes^n\mathcal{H})$ and by $\Lambda \mathcal{H}$ the Hilbert space $P(T(\mathcal{H}))=\bigoplus_{n=0}^\infty\Lambda^n\mathcal{H}$.

For $f\in \mathcal{H}$ we define a linearly map $a(f)$ on $\Lambda\mathcal{H}$ given by $a(f):\Lambda^n\mathcal{H}\to \Lambda^{n+1}\mathcal{H}$
$$a(f)(\xi_1\wedge\cdots\wedge \xi_n)=f\wedge \xi_1\wedge\cdots\wedge \xi_n$$
The operators $a(f),a(g)$ satisfy Canonical Anticommutation Relation (CAR):
\begin{equation}
\begin{split}
a(f)a(g)+a(g)a(f)&=0\\
a(f)a(g)^*+a(g)^*a(f)&=\langle f,g\rangle 1
\end{split}
\end{equation}

Then the CAR algebra $CAR(\mathcal{H})$ the *-algebra generated by $a(f)$ represented on $\Lambda\mathcal{H}$. We have that $CAR(\mathcal{H})$ is a C$^*$ algebra when $\mathcal{H}$ is a Hilbert space. When $dim\mathcal{H}=1$, one see that $CAR(\mathcal{H})$ is isomorphic to $M_2(\mathbb{C})$. In this case $a(\xi)$ is a partial isometry from $\mathbb{C}1$ onto $\mathbb{C}\xi$ and $\Lambda\mathcal{H}$ is 2-dimensional Hilbert space spanned by $1,\xi$.

For any unitary operator $U$ on $\mathcal{H}$, $\alpha_U$ is an automorphism of $CAR(\mathcal{H})$ given by $\alpha_U(a(f))=a(Uf)$ for any $f$ in $\mathcal{H}$. The automorphism $\alpha_U$ arising from a unitary operator $U$ on $\mathcal{H}$ is called Bogoliubov automorphism of $CAR(\mathcal{H})$.

Let $T\in\mathcal{B}(\mathcal{H})$ and $0\leq T\leq 1$. The quasi-free state $\varphi_T$ on $CAR(\mathcal{H})$ given by
$$\varphi_T(a(g_m)^*\cdots a(g_1)^*a(f_1)\cdots a(f_n))=\delta_{m,n}det(\langle Tf_i,g_j\rangle)_{i,j=1}^n$$ and $\varphi_T(1)=1$

\begin{lemma}\label{C1}
Let $U$ be a unitary operator on a Hilbert space $\mathcal{H}$ whose spectrum has absolutely continuous part. Then $\mu$ is not (weakly) linearly disjoint from $(\mathcal{B}(\mathcal{H}),Ad(U))$ and $\mu$ is not linearly disjoint from $(CAR(\mathcal{H}),\alpha_U)$.
\end{lemma}
\begin{proof}
Suppose that the absolutely continuous part $U_a$ of $U$ acts on the subspace $\mathcal{H}_a$ of $\mathcal{H}$.

We can write $\mathcal{H}_a=\int_\mathbb{T}^\oplus \mathcal{H}_zd\lambda(z)$ and
$U_a=\int_\mathbb{T}^\oplus zd\lambda(z)$ where $\lambda$ is the Lebesgue measure on $\mathbb{T}$.

Let $\xi=\int_\mathbb{T}^\oplus \xi_z d\lambda(z)$ be a unit vector in $\mathcal{H}_a$. We then define $\xi_k=\int_\mathbb{T}^\oplus z^k\xi_zd\lambda(z)$ for $k\in\mathbb{Z}$. By direct computation, $\{\xi_k\}_{k\in\mathbb{Z}}$ is orthogonal family of unit vectors in $\mathcal{H}_a$.

Let $P_k$ be the projection of $\mathcal{H}$ onto $\mathbb{C}\xi_k$ and $T=\sum_{k=1}^\infty \mu(k)P_{-k}$. Then $T$ is bounded.

One can checks that $U_a^nP_k{U_a^*}^n=P_{k+n}$ for any $k,n$ in $\mathbb{Z}$. Then we see that
$$\sum_{n\leq N}\mu(n)\langle U_a^nT{U_a^*}^n\xi,\xi\rangle=\sum_{n\leq N}|\mu(n)|.$$
It is known that $\frac{1}{N}\sum_{n\leq N}|\mu(n)|$ does not converge to zero. Hence $\mu$ is not (weakly) linearly disjoint from $(\mathcal{B}(\mathcal{H}),Ad(U))$.

For CAR algebra $CAR(\mathcal{H})$, we consider the quasi-free state $\varphi_{(T+I)/2}$. Then
\begin{equation*}
\begin{split}
\varphi_{(T+I)/2}(\alpha_U^n(a(\xi)^*a(\xi)))&=\varphi_{(T+I)/2}(a(U^n\xi)^*a(U^n\xi))\\
&=\langle \frac{T+I}{2}U^n\xi, U^n\xi \rangle\\
&=\frac{1}{2}(\mu(n)+1),
\end{split}
\end{equation*}
and
$$\sum_{n\leq N}\mu(n)\varphi_{(T+I)/2}(\alpha_U^n(a(\xi)^*a(\xi)))=\frac{1}{2}\sum_{n\leq N}(|\mu(n)|+\mu(n)).$$
Hence $\mu$ is not linearly disjoint from $(CAR(\mathcal{H}),\alpha_U)$.
\end{proof}

By Lemma \ref{C1}, we will study unitary operator $U$ on a Hilbert space $\mathcal{H}$ whose spectrum measure is singular. The techniques here benefits a lot from Lemma 5.1 in \cite{SV}.

Now we need to find the relation between that M\"{o}bius disjointness for $\mathcal{B}(\mathcal{H})$ with inner automorphism and that for $CAR(\mathcal{H})$ with Bogoliubov automorphism.

\begin{proposition}\label{CAR1}
Let $U$ be a unitary operator on a Hilbert space $\mathcal{H}$ and $\widetilde{U}=U\oplus I$ on Hilbert space $\mathcal{H}\oplus \mathbb{C}\xi_0(=\widetilde{\mathcal{H}})$ for some unit vector $\xi_0$.  Then $\mu$ is linearly disjoint from $(CAR(\mathcal{H}),\alpha_U)$ if $\mu$ is weakly disjoint from $(\mathcal{B}(\widetilde{\mathcal{H}}^{\otimes m}), Ad(\widetilde{U}^{*\otimes m})$ for any $m\geq 1$. In particular, when $U$ has pure point spectrum $1$, then $\mu$ is linearly disjoint from $(CAR(\mathcal{H}),\alpha_U)$ if $\mu$ is weakly disjoint from $(\mathcal{B}(\mathcal{H}^{\otimes m}), Ad(U^{*\otimes m})$ for any $m\geq 1$.
\end{proposition}

\begin{proof}
Suppose that $\mu$ is weakly linearly disjoint from $(\mathcal{B}(\widetilde{\mathcal{H}}^{\otimes m}),Ad(\widetilde{U}^{*\otimes m})$.

Since the linearly span of
$$a(\xi_{1})\cdots a(\xi_{m}), m\geq 1, a(\eta_{k})^*\cdots a(\eta_{1})^*, k\geq 1$$
and $$a(\eta_{k})^*\cdots a(\eta_{1})^*a(\xi_{1})\cdots a(\xi_{m}), k,m\geq 1$$ is norm dense in $CAR(\mathcal{H})$, we will consider element $A$ in $CAR(\mathcal{H})$ in the three forms as above.

When $A=a(\xi_{1})\cdots a(\xi_{m})$, we have
\begin{equation*}
\begin{split}
\|\frac{1}{N}\sum_{n\leq N}\mu(n)\alpha_U^n(A)\|
&=\|\frac{1}{N}\sum_{n\leq N}\mu(n) a(U^n\xi_{1})\cdots a(U^n\xi_{m}))\|\\
&=\|\frac{1}{N}\sum_{n\leq N}\mu(n) a(U^n\xi_{1}\wedge \cdots \wedge U^n\xi_{m})\|\\
&=\|\frac{1}{N}\sum_{n\leq N}\mu(n)(U^n\otimes\cdots\otimes U^n)(\xi_{1}\wedge \cdots \wedge\xi_{m})\|\\
&\leq \|\frac{1}{N}\sum_{n\leq N}\mu(n)(U^n\otimes\cdots\otimes U^n)\|\|\xi_{1}\wedge \cdots \wedge\xi_{m}\|\\
&\leq \frac{1}{N}\max_{z\in sp(U\otimes\cdots\otimes U)}|\sum_{n\leq N}\mu(n)z^n|\|\xi_{1}\wedge \cdots \wedge\xi_{m}\|\to 0.
\end{split}
\end{equation*}

When $A=a(\eta_{k})^*\cdots a(\eta_{1})^*$, we consider its adjoint $A^*$ and obtain that $\|\frac{1}{N}\sum_{n\leq N}\mu(n)\alpha_U^n(A)\|\to 0$ as $N\to\infty$.

When $A=a(\eta_{k})^*\cdots a(\eta_{1})^*a(\xi_{1})\cdots a(\xi_{m})$, we assume that $k\leq m$. If $k>m$, we will consider $A^*$ instead of $A$. Let $\rho$ be a state on $CAR(\mathcal{H})$. Then
\begin{equation*}
\begin{split}
\rho(\alpha_U^n(A))&=\rho(\alpha_U^n(a(\eta_{k})^*\cdots a(\eta_{1})^*a(\xi_{1})\cdots a(\xi_{j_m})))\\
&=\rho(a(U^n\eta_{k})^*\cdots a(U^n\eta_{1})^*a(U^n\xi_{1})\cdots a(U^n\xi_{m}))\\
&=\rho(a(U^n\eta_{1}\wedge \cdots \wedge U^n\eta_{k})^*a(U^n\xi_{1}\wedge \cdots \wedge U^n\xi_{m}))).
\end{split}
\end{equation*}

Let $P_n$ be the projection of $\otimes^n\widetilde{\mathcal{H}}$ onto $\otimes^n\mathcal{H}$. Let $V$ be transformation from $\otimes^k\widetilde{\mathcal{H}}$ into $\otimes^m\widetilde{\mathcal{H}}$ given by $V\xi=\xi\otimes \xi_0\otimes \cdots \otimes \xi_0$ for any $\xi$ in $\otimes^k\widetilde{\mathcal{H}}$.

For any $\xi$ in $\otimes^m\widetilde{\mathcal{H}}$ and $\eta$ in $\otimes^m\widetilde{\mathcal{H}}$, let $B(\eta,\xi)=\rho(a(PP_kV^*\eta)^*a(PP_m\xi))$. Then $B(\cdot,\cdot)$ is linearly in the first variable and conjugate-linearly in the second variable. Moreover $B(\cdot,\cdot)$ is bounded. Hence there exists a bounded element $T$ in $\mathcal{B}(\otimes^m\widetilde{\mathcal{H}})$ such that $\rho(a(PP_kV^*\eta)^*a(PP_m\xi))=\langle T\xi,\eta \rangle$ for any $\xi,\eta$ in $\otimes^m\widetilde{\mathcal{H}}$.

Let $\eta=\eta_{1}\wedge\cdots \wedge \eta_{k}\otimes \xi_0\otimes\cdots\otimes \xi_0$ and $\xi=\xi_{1}\wedge\cdots \wedge \xi_{m}$. Then $PP_m\widetilde{U}\otimes\cdots \otimes\widetilde{U}\xi=\widetilde{U}\otimes\cdots \otimes\widetilde{U}\xi$ and $PP_kV^*\widetilde{U}\otimes\cdots \otimes\widetilde{U}\eta=PP_kU^n\eta_{1}\wedge\cdots \wedge U^n\eta_{k}=U^n\eta_{1}\wedge\cdots \wedge U^n\eta_{k}$. On the other hand,
$$
\rho(\alpha_U^n(A))=\langle \widetilde{U}^{*n}\otimes\cdots\otimes \widetilde{U}^{*n}T\widetilde{U}^n\otimes\cdots \otimes\widetilde{U}^n\xi,\eta \rangle
$$
Hence by the assumption, we have that $\frac{1}{N}\sum_{n\leq N}\mu(n)\rho(\alpha_U^n(A)) \to 0$ for any element $A$ in the linearly span of  $a(\xi_{1})\cdots a(\xi_{m}), m\geq 1$, $a(\eta_{k})^*\cdots a(\eta_{1})^*, k\geq 1$ and $a(\eta_{k})^*\cdots a(\eta_{i_1})^*a(\xi_{1})\cdots a(\xi_{m}), k,m\geq 1$.

Therefore $\mu$ is linearly disjoint from $(CAR(\mathcal{H}),\alpha_U)$.

If $U$ has pure point spectrum $1$, we let $\xi_0'$ be its unit eigenvector. Then we define $V'$ to be transformation from $\otimes^k\mathcal{H}$ into $\otimes^m\mathcal{H}$ by $V'\xi=\xi\otimes\xi_0'\otimes\cdots\otimes\xi_0'$. Similarly, we have
$$
\rho(\alpha_U^n(A))=\langle {U}^{*n}\otimes\cdots\otimes {U}^{*n}T {U}^n\otimes\cdots \otimes {U}^n\xi,\eta \rangle.
$$
Then the conclusion follows by a similar argument above.
\end{proof}

\begin{remark}
It is clear that if $\mu$ is linearly disjoint from $(CAR(\mathcal{H}),\alpha_U)$, then $\mu$ is weakly linearly disjoint from $(\mathcal{B}(\mathcal{H}),Ad(U))$. \end{remark}

The CNT entropy of Bogoliubov automorphism was first computed by E. St{\o}rmer and  D. Voiculescu \cite{SV}. Later in \cite{Nesh01}. In \cite{SN}, for the Bogoliubov automorphism $\alpha_U$ of CAR algebra $CAR(\mathcal{H})$, the topology entropy $ht(\alpha_U)=\log 2\int_{\mathbb{T}}m_U(z)d\mu(z)$. In particular, $ht(\alpha_U)=0$ if $U$ has singular spectral measure.

\begin{proposition}\label{CAR2}
The M\"{o}bius function $\mu$ is linearly disjoint from $(CAR(\mathcal{H}),\alpha_U)$ when the spectral measure of $U$ is pure point measure.
\end{proposition}
\begin{proof}
Since the spectral measure of $U$ is pure point measure, by Lemma \ref{state}, we estimate 
$$\frac{1}{N}\sum_{n=1}^N\mu(n)\langle U^nT{U^*}^n\xi,\xi\rangle$$
for $(\mathcal{B}(\mathcal{H}),Ad(U))$, where $\xi$ is a unit eigenvector corresponding to some eigenvalue. It is easy to see that the sum goes to zero as $N$ goes to infinity. By Proposition \ref{CAR1}, we see that $\mu$ is linear disjoint from $(CAR(\mathcal{H}),\alpha_U)$.
\end{proof}

\begin{question}
Is $\mu$ linearly disjoint from $(CAR(\mathcal{H}),\alpha_U)$ (or is $\mu$ weakly linearly disjoint from $(\mathcal{B}(\mathcal{H}),Ad(U))$), when the spectral measure of $U$ is singular continuous?
\end{question}

\begin{remark}
If the question above has an affirmative answer (which we believe), then $\mu$ is linear disjoint from $(CAR(\mathcal{H}),\alpha_U)$ if and only if $\mu$ is weakly linearly disjoint from $(\mathcal{B}(\mathcal{H}),Ad(U))$.
\end{remark}

\section{Free Group C$^*$ Algebras}

Let $\mathcal{F}_\mathbb{Z}$ be the free group on generators $\ldots,g_{-1},g_0,g_1,\ldots$ and $C_r^*(F_\mathbb{Z})$ the reduced free group C$^*$ algebra. Let $\alpha$ be the automorphism of $C_r^*(F_\mathbb{Z})$ such that $\alpha(g_i)=g_{i+1}$.

\begin{proposition}
The M\"{o}bious function $\mu$ is linearly disjoint from $(C_r^*(F_\mathbb{Z}),\alpha)$ whose Voiculescu-Brown entropy $ht(\alpha)=0$.
\end{proposition}
\begin{proof}
Let $\hat{m}=(i_1,\ldots,i_k)$ when $k\geq 1$ and $\hat{m}=\emptyset$ when $k=0$, where $i_1,\ldots,i_k\in\mathbb{Z}$. Denote by $g_{\hat{m}}=g_{i_1}\cdots g_{i_k}$, $g_{\emptyset}=e$ and $|\hat{m}|$ the length of $\hat{m}$.

For any given $g_{\hat{m}}$, we let $l=\max\{|i_1|,\ldots,|i_k|\}$.

Let
$$B_{k,N}=\sum_{p=0}^{[N/(2l+1)]-1}\mu(p(2l+1)+k)\alpha^{p(2l+1)+k}(g_{\hat{m}})$$ for $k=1,\ldots,2l+1$.

It is clear that $B_{k,N}$ is sum of free elements, i.e. $B_{k,N}=D_{1,k,N}+\cdots+D_{q,k,N}$, where $q\leq [N/(2l+1)]$ and $D_{1,k,N},\ldots,D_{q,k,N}$ are free. We see that $\frac{1}{2}(D_{1,k,N}+D_{1,k,N}^*),\ldots,\frac{1}{2}(D_{q,k,N}+D_{q,k,N}^*)$ are free. By the central limit theorem in \cite{VDN}, we obtain that $\frac{1}{\sqrt{q}}(\frac{1}{2}(D_{1,k,N}+D_{1,k,N}^*)+\cdots+\frac{1}{2}(D_{1,k,N}+D_{1,k,N}^*))$ converges in distribution to a semicircle element. Then $\|(2l+1)B_{k,N}/N\|\to 0$ as $N\to \infty$ and hence there exists $N_0$ in $\mathbb{N}$ such that $\|(2l+1)B_{k,N}/N\|<\epsilon$ for $k=1,\ldots, 2l+1$ and $N>N_0$. Therefore
\begin{equation*}
|\frac{1}{N}\sum_{n\leq N}\mu(n)\rho(\alpha^ng_{\hat{m}}))|=|\frac{1}{N}\sum_{k=1}^{2l+1}\rho(B_{k,N})|<\epsilon
\end{equation*}
This shows that $\mu$ is linearly disjoint from $(C_r^*(F_\mathbb{Z}),\alpha)$. By \cite{Choda06}, the Voiculescu-Brown entropy $ht(\alpha)=0$.
\end{proof}



$$$$

\noindent Jinsong Wu\\
University of Science and Technology of China\\
Hefei, Anhui, China 230026\\
wjsl@ustc.edu.cn\\
\\
Wei Yuan\\
Academy of Mathematics and Systems Science, CAS\\
Beijing, China 100190\\
wyuan@math.ac.cn\\

\end{document}